\documentclass[11pt]{amsart}
\usepackage[T1]{fontenc}
\usepackage[utf8]{inputenc}
\usepackage{lmodern}
\usepackage[margin=1in]{geometry}
\usepackage{amsmath,amssymb,amsthm,mathtools,booktabs,microtype}
\usepackage[colorlinks=true,linkcolor=blue,citecolor=blue,urlcolor=blue]{hyperref}
\hypersetup{pdftitle={A Counterexample to a Sparse Basis Conjecture of Brualdi, Friedland, and Pothen},pdfauthor={Mohsen Aliabadi}}

\newtheorem{theorem}{Theorem}[section]
\newtheorem{proposition}[theorem]{Proposition}

\newtheorem{corollary}[theorem]{Corollary}
\newtheorem{conjecture}[theorem]{Conjecture}
\theoremstyle{remark}

\newcommand{\Q}{\mathbb Q}
\newcommand{\Row}{\operatorname{row}}
\newcommand{\supp}{\operatorname{supp}}
\DeclareMathOperator{\rank}{rank}
\DeclareMathOperator{\Span}{span}

\title[A counterexample to a sparse basis conjecture]{A Counterexample to a Sparse Basis Conjecture of Brualdi, Friedland, and Pothen}
\author{Mohsen Aliabadi}
\address{Department of Mathematics, Clayton State University, Morrow, GA 30260, USA}
\email{maliabadi@clayton.edu}
\date{}

\begin{document}
\begin{abstract}
Brualdi, Friedland, and Pothen conjectured that certain rank inequalities characterize bases consisting of elementary vectors in the row space of a matrix whose nonzero entries are algebraically independent. We give a $4\times6$ counterexample: all the inequalities hold, but the four elementary vectors form a minimally dependent family. A construction using two planes in a common three-dimensional subspace explains the dependence and gives a second counterexample pattern with only four prescribed zeros. We also prove that the proposed criterion is valid over every field when the number of columns exceeds the row rank by at most one. A direct-sum extension gives counterexamples of size $m\times(m+2)$ for every $m\ge4$. Thus the minimum column excess permitting a counterexample is exactly two in every such rank.
\end{abstract}
\maketitle

\section{Introduction}

The sparse basis problem asks for a basis of a given subspace with as few nonzero coordinates as possible. A basic difficulty is to ensure that individually sparse vectors remain linearly independent when assembled into a basis. This issue occurs in algorithms for sparse null spaces, including the constructions of Coleman and Pothen~\cite{CP}. Subsequent work of Gottlieb and Neylon~\cite{GN} relates sparse null space computation to matrix sparsification and studies approximation and algorithmic questions. These optimization problems provide the broader setting for the independence criterion considered here.

Brualdi, Friedland, and Pothen~\cite{BFP} studied \emph{elementary vectors}, that is, nonzero vectors whose supports are minimal under inclusion. Their Theorems~1.1 and~1.2 characterize independence of such vectors in the row and null spaces of a \emph{fully generic} matrix, meaning that all entries are nonzero and algebraically independent over $\Q$. In that setting, the characterization is combinatorial: the complements of the supports satisfy suitable intersection inequalities. At the end of their paper they proposed an extension to matrices with prescribed zeros and algebraically independent nonzero entries. The extension replaces cardinalities of intersections by ranks of the corresponding column submatrices.

The present note disproves that extension. Our matrix has four rows and six columns. Four elementary vectors in its row space satisfy every proposed rank inequality, yet have a unique linear dependence up to scaling, with all four coefficients nonzero. We then give a geometric construction that produces the same failure for a family of configurations and, in particular, for a zero pattern having only four prescribed zeros. Finally, we prove that the criterion holds whenever $n-m\le1$, without a genericity assumption, and extend the counterexample to size $m\times(m+2)$ for every $m\ge4$.

Our construction uses the characterization through geometric intersections of hyperplanes in~\cite[Corollary~7.2 and Theorem~7.4]{BFP}. The counterexample concerns the proposed replacement of those intersections by ranks of common indexed columns. It leaves the fully generic theorems intact. In particular, it shows that the proposed rank inequalities cannot serve as a sufficient independence certificate for elementary vectors of arbitrary sparse generic matrices.

\section{The conjecture}\label{sec:framework}

Let $A\in F^{m\times n}$ have rank $m$, where $F$ is a field. Write $[n]=\{1,\ldots,n\}$, denote the columns by $A_1,\ldots,A_n$, and put
\[
 W=\Row(A)=\{y^TA:y\in F^m\},\qquad
 U(J)=\Span_F\{A_j:j\in J\}\quad(J\subseteq[n]).
\]
The matrix $A[:,J]$ consists of the columns indexed by $J$; its rank is zero when $J=\varnothing$. For $x\in F^n$, set
\[
 \supp(x)=\{j:x_j\ne0\},\qquad Z(x)=[n]\setminus\supp(x).
\]
A nonzero vector $x\in W$ is \emph{elementary} if no nonzero vector of $W$ has support properly contained in $\supp(x)$. Thus its zero set is maximal under inclusion among zero sets of nonzero vectors of $W$.

We use two elementary facts from the framework of~\cite[Lemma~7.3, Corollary~7.2, and Theorem~7.4]{BFP}. The original statements concern real spaces; the following verification also applies over an arbitrary field.

\begin{proposition}\label{prop:background}
Let $A$ have full row rank.
\begin{enumerate}
\item A set $J\subseteq[n]$ is the zero set of an elementary vector if and only if $U(J)$ is a hyperplane of $F^m$ and $A_r\notin U(J)$ for every $r\notin J$.
\item If $x_s=y_s^TA$ are elementary vectors with zero sets $J_s$ and $U_s=U(J_s)$, then, for every nonempty index set $P$,
\begin{equation}\label{eq:geometric}
 \dim\bigcap_{s\in P}U_s
 =m-\dim\Span_F\{y_s:s\in P\}
 =m-\dim\Span_F\{x_s:s\in P\}.
\end{equation}
Consequently, $x_1,\ldots,x_t$ are independent if and only if
$\dim\bigcap_{s\in P}U_s\le m-|P|$ for every nonempty $P\subseteq[t]$.
\end{enumerate}
\end{proposition}

\begin{proof}
Suppose first that $x=y^TA$ is elementary and $J=Z(x)$. Then $U(J)\subseteq\ker y^T$. If $\dim U(J)\le m-2$, choose a column $A_r$ outside $U(J)$. The span of $U(J)$ and $A_r$ has dimension at most $m-1$, so it has a nonzero annihilating functional $z^T$. Full row rank gives $z^TA\ne0$, whereas this vector vanishes on $J\cup\{r\}$, a contradiction. Hence $\dim U(J)=m-1$, and the definition of $J$ excludes every other column from $U(J)$.

Conversely, a nonzero normal $y$ to such a hyperplane gives $Z(y^TA)=J$. Any row-space vector vanishing on $J$ and on one additional coordinate annihilates a spanning set of $F^m$, and is therefore zero. This proves the first assertion. For the second, $U_s=\ker y_s^T$, and the map $y\mapsto y^TA$ is an isomorphism from $F^m$ onto $W$. The rank-nullity theorem gives~\eqref{eq:geometric}.
\end{proof}

A matrix over an extension of $\Q$ is called \emph{sparse generic} relative to its zero pattern when all its displayed nonzero entries are algebraically independent over $\Q$. We state the conjecture from the end of~\cite{BFP} using zero sets instead of complements of supports.

\begin{conjecture}[Brualdi--Friedland--Pothen]\label{conj:BFP}
Let $A$ be a sparse generic $m\times n$ matrix of rank $m$, and let $x_1,\ldots,x_m$ be elementary vectors of $\Row(A)$. Put $J_s=Z(x_s)$. Then these vectors form a basis if and only if
\begin{equation}\label{eq:ranktest}
 \rank A\left[:,\bigcap_{s\in P}J_s\right]\le m-|P|
 \qquad(\varnothing\ne P\subseteq[m]).
\end{equation}
\end{conjecture}

Necessity holds for every full-row-rank matrix, since
\begin{equation}\label{eq:inclusion}
 U\left(\bigcap_{s\in P}J_s\right)\subseteq\bigcap_{s\in P}U(J_s),
\end{equation}
and~\eqref{eq:geometric} applies. We disprove sufficiency. For fully generic matrices, BFP's theorem supplies the nontrivial link between the combinatorial and geometric conditions. Proposition~\ref{prop:background} by itself does not provide that link or a new proof of their fully generic results.

\subsection{A six-column counterexample}\label{sec:example}

Let $K=\Q(a,b,e,f,g,h,i,j,k,l)$ be the rational function field in ten independent indeterminates, and consider
\begin{equation}\label{eq:A}
 A=\begin{pmatrix}
 a&0&0&0&0&k\\
 0&e&0&h&0&l\\
 b&0&0&i&j&0\\
 0&f&g&0&0&0
 \end{pmatrix}.
\end{equation}
Every displayed nonzero entry is an indeterminate, so $A$ is sparse generic. Moreover,
$\det A[:,\{1,2,3,5\}]=-aegj\ne0$.

\begin{theorem}\label{thm:example}
The row space of~\eqref{eq:A} contains four elementary vectors $x_1,x_2,x_3,x_4$ with zero sets
\begin{equation}\label{eq:zeros}
 J_1=\{3,5,6\},\quad J_2=\{1,3,4\},\quad
 J_3=\{1,2,4\},\quad J_4=\{2,5,6\}.
\end{equation}
They satisfy all inequalities~\eqref{eq:ranktest}, but form a minimally dependent family: their span has dimension three and every proper subfamily is independent.
\end{theorem}

\begin{proof}
Put $\Delta=ail+bhk$, which is nonzero in $K$, and define
\[
 Y=\begin{pmatrix}
 -l&k&0&0\\
 -bh&-ai&ah&0\\
 -bfh&-afi&afh&aei\\
 fl&-fk&0&ek
 \end{pmatrix}.
\]
Writing the rows of $Y$ as $y_s^T$, set $x_s=y_s^TA$. Direct multiplication gives
\begin{equation}\label{eq:X}
 X=YA=\begin{pmatrix}
 -al&ek&0&hk&0&0\\
 0&-aei&0&0&ahj&-\Delta\\
 0&0&aeig&0&afhj&-f\Delta\\
 afl&0&ekg&-fhk&0&0
 \end{pmatrix}.
\end{equation}
All displayed nonzero entries of $X$ are nonzero elements of $K$, so its rows have the zero sets~\eqref{eq:zeros}.

Each column set $J_s$ has rank three. For completeness, the following minors certify this:
\[
\begin{array}{c|c|c}
 s&\text{row indices of a minor of }A[:,J_s]&\text{determinant}\\ \hline
 1&\{1,3,4\}&-gjk\\
 2&\{1,2,4\}&-agh\\
 3&\{1,2,3\}&aei\\
 4&\{1,2,3\}&ejk
\end{array}
\]
Since the zero sets are exact, every column outside $J_s$ lies outside $\ker y_s^T=U(J_s)$. Proposition~\ref{prop:background} shows that all four rows of $X$ are elementary.

For a singleton $P$, the left side of~\eqref{eq:ranktest} equals $3$. For pairs, the complete verification is
\[
\begin{array}{c|c|c}
 P&\bigcap_{s\in P}J_s&\rank A[:,\bigcap_{s\in P}J_s]\\ \hline
 \{1,2\}&\{3\}&1\\
 \{1,3\}&\varnothing&0\\
 \{1,4\}&\{5,6\}&2\\
 \{2,3\}&\{1,4\}&2\\
 \{2,4\}&\varnothing&0\\
 \{3,4\}&\{2\}&1
\end{array}
\]
Thus every pair gives rank at most $2$. Every intersection of at least three of the $J_s$ is empty, which proves the remaining inequalities.

On the other hand,
\begin{equation}\label{eq:dependence}
 afi\,x_1+fk\,x_2-k\,x_3+ai\,x_4=0.
\end{equation}
The same relation holds among the rows of $Y$. The minor of $Y$ on rows $1,2,4$ and columns $2,3,4$ has determinant $aehk^2\ne0$. Hence $\rank Y=\rank X=3$. The space of relations is therefore one-dimensional, and every coefficient in~\eqref{eq:dependence} is nonzero. A dependence of a proper subfamily would give a relation with a zero coefficient, so no such dependence exists.
\end{proof}

The use of a rational function field makes the nonvanishing assertions explicit. Equally, the indeterminates can be replaced by algebraically independent real numbers, giving a counterexample within the real setting of the original conjecture.

\section{A geometric construction and a zero-pattern family}\label{sec:construction}

The following construction identifies conditions on the columns that force the failure. It uses the established hyperplane criterion, but supplies a configuration in which its geometric intersection is necessarily larger than the span of the common indexed columns.

\begin{theorem}\label{thm:construction}
Let $V$ be a four-dimensional vector space over a field $F$, and let $H\subset V$ have dimension three. Choose
\[
 p_1,p_2,q_1,q_2\in H,\qquad u,v\in V\setminus H,
\]
such that every three of $p_1,p_2,q_1,q_2$ are independent and both lists
\begin{equation}\label{eq:extension}
 (p_1,p_2,u,v),\qquad(q_1,q_2,u,v)
\end{equation}
are bases of $V$. In any coordinates on $V$, form
$C=[\,p_1\ p_2\ q_1\ q_2\ u\ v\,]$. Then $\Row(C)$ has elementary vectors with zero sets
\begin{equation}\label{eq:rectangle}
 \{1,2,5\},\quad\{1,2,6\},\quad
 \{3,4,5\},\quad\{3,4,6\}.
\end{equation}
These four vectors satisfy the rank-intersection inequalities for $m=4$ and are minimally dependent.
\end{theorem}

\begin{proof}
Set $P=\Span(p_1,p_2)$ and $Q=\Span(q_1,q_2)$. The hypotheses give $P+Q=H$, so $L=P\cap Q$ has dimension one. Consider the four hyperplanes
\[
 P+Fu,\quad P+Fv,\quad Q+Fu,\quad Q+Fv.
\]
Since $u,v\notin H$, their intersections with $H$ are respectively $P,P,Q,Q$. Thus neither $q_1$ nor $q_2$ belongs to either of the first two hyperplanes, and neither $p_1$ nor $p_2$ belongs to either of the last two. The basis conditions~\eqref{eq:extension} also exclude $v$ from the hyperplanes formed with $u$, and $u$ from those formed with $v$. Proposition~\ref{prop:background} therefore gives elementary vectors with exactly the zero sets~\eqref{eq:rectangle}.

Each singleton zero set has rank three. Intersections of two zero sets consist of either one of the independent pairs $\{1,2\}$ and $\{3,4\}$, a single index, or the empty set. All intersections of three or four zero sets are empty. Hence the rank-intersection inequalities hold.

The first two hyperplanes intersect in $P$, and the last two intersect in $Q$, by~\eqref{eq:extension}. Their total intersection is consequently $L$. Any three include both hyperplanes based on $P$ or both based on $Q$; intersecting with the remaining hyperplane again gives $L$. Formula~\eqref{eq:geometric} now shows that the four elementary vectors have rank three and that every three are independent.
\end{proof}

This construction applies to a pattern that is substantially less sparse than~\eqref{eq:A}.

\begin{corollary}\label{cor:pattern}
Let the displayed entries of
\begin{equation}\label{eq:pattern}
 C=\begin{pmatrix}
 c_{11}&c_{12}&c_{13}&c_{14}&c_{15}&c_{16}\\
 c_{21}&c_{22}&c_{23}&c_{24}&c_{25}&c_{26}\\
 c_{31}&c_{32}&c_{33}&c_{34}&c_{35}&c_{36}\\
 0&0&0&0&c_{45}&c_{46}
 \end{pmatrix}
\end{equation}
be algebraically independent over $\Q$. Then the four zero sets~\eqref{eq:rectangle} determine a counterexample to Conjecture~\ref{conj:BFP}.
\end{corollary}

\begin{proof}
Take $H=\{w\in K^4:w_4=0\}$, where $K$ is the rational function field in the twenty displayed entries. Every three of the first four columns are independent because the corresponding $3\times3$ determinant in the first three rows is a nonzero polynomial. Columns $5$ and $6$ lie outside $H$. The determinants on columns $\{1,2,5,6\}$ and $\{3,4,5,6\}$ are also nonzero polynomials: respectively they contain the monomials $c_{11}c_{22}c_{35}c_{46}$ and $c_{13}c_{24}c_{35}c_{46}$ with nonzero coefficients. Distinct determinant terms have distinct monomials, so these terms cannot cancel. Theorem~\ref{thm:construction} applies.
\end{proof}

The explicit matrix~\eqref{eq:A} fits the same construction after ordering its columns as
\[
 (p_1,p_2,q_1,q_2,u,v)=(A_1,A_4,A_5,A_6,A_3,A_2).
\]
Here the four determinants for triples in $H$ are $ahj$, $-\Delta$, $-ajl$, and $hjk$, and the two determinants in~\eqref{eq:extension} are $aegi$ and $-egjk$. Thus all hypotheses hold. In particular, the common line of the four hyperplanes is forced by two planes in $H$; it is not merely an alternative certificate obtained after computing the dependence. Corollary~\ref{cor:pattern} shows that four prescribed zero entries already suffice to force this configuration. No minimality assertion about the number of prescribed zeros is intended.

\section{Sharpness in the number of columns}\label{sec:sharpness}

The six-column construction is optimal for four rows. We first establish a more general sufficient setting for the conjectured inequalities.

\begin{proposition}\label{prop:small-excess}
Let $A\in F^{m\times n}$ have rank $m$ over an arbitrary field, and suppose $n-m\le1$. A list of $m$ elementary vectors of $\Row(A)$ is a basis if and only if it satisfies~\eqref{eq:ranktest}.
\end{proposition}

\begin{proof}
Necessity was proved after Conjecture~\ref{conj:BFP}. If $n=m$, the elementary vectors are nonzero multiples of the coordinate vectors. A repeated direction violates~\eqref{eq:ranktest} for that pair, because its common zero set has column rank $m-1$. Thus the condition forces a basis.

Suppose $n=m+1$, and choose $0\ne z\in\ker A$. Set $S=\supp(z)$ and $T=[n]\setminus S$. Since
\[
 \Row(A)=\left\{x\in F^n:\sum_{i\in S}z_ix_i=0\right\},
\]
its elementary vectors, up to nonzero multiples, are precisely
\begin{equation}\label{eq:codim-one-vectors}
 e_t\quad(t\in T),\qquad
 z_i^{-1}e_i-z_j^{-1}e_j\quad(i,j\in S,\ i\ne j),
\end{equation}
where $e_i$ denotes the $i$th coordinate vector. Indeed, a vector involving a coordinate in $T$ contains the support of the corresponding $e_t$, while a nonzero vector supported in $S$ needs at least two coordinates and contains the support of a two-coordinate vector in~\eqref{eq:codim-one-vectors}.

Multiplying coordinate $i\in S$ by $z_i$ transforms the latter vectors into $e_i-e_j$. Regard each chosen two-coordinate vector as an edge of a multigraph on $S$. The edge vectors of a forest are independent, as one sees by removing a leaf, whereas the vectors around a cycle are dependent. Parallel edges give a cycle of length two. The coordinate directions indexed by $T$ form a separate direct summand. Consequently, any dependent list contains either a repeated direction $e_t$, or a cycle consisting of $q$ edges on $q$ vertices of $S$.

In the first case, select the repeated pair. Its common zero set is $R=[n]\setminus\{t\}$. The only dependence among all columns of $A$ is a multiple of $z$, whose support is contained in $R$. Hence $\rank A[:,R]=|R|-1=m-1$, contradicting the required bound $m-2$.

In the second case, let $P$ index the $q$ selected edge vectors and let $D\subseteq S$ be their $q$ vertices. Their common zero set is $R=[n]\setminus D$. The columns indexed by $R$ are independent: a dependence supported there would extend to a multiple of $z$, which is nonzero on $D$. Therefore
\[
 \rank A[:,R]=|R|=m+1-q>m-q=m-|P|,
\]
again contradicting~\eqref{eq:ranktest}. Thus a list satisfying all the inequalities is independent.
\end{proof}

The obstruction also persists in every larger row rank.

\begin{corollary}\label{cor:all-ranks}
For each $m\ge4$, there is a sparse generic $m\times(m+2)$ matrix with $m$ elementary row-space vectors satisfying~\eqref{eq:ranktest} whose span has dimension $m-1$. For every such $m$, the least possible number of columns in a counterexample is $m+2$.
\end{corollary}

\begin{proof}
Write $r=m-4$ and let $A$ be~\eqref{eq:A}. Adjoin independent indeterminates $d_1,\ldots,d_r$ and form
\[
 \widetilde A=\begin{pmatrix}A&0\\0&D\end{pmatrix},
 \qquad D=\operatorname{diag}(d_1,\ldots,d_r),
\]
with the second block omitted if $r=0$. This matrix is sparse generic and has rank $m$. Extend the four vectors of Theorem~\ref{thm:example} by zeros, and adjoin the $r$ coordinate vectors on the new columns. The resulting vectors are elementary because the row space is a direct sum, and their span has dimension $3+r=m-1$.

Consider a nonempty subfamily containing $a$ of the original four vectors and $b$ of the new coordinate vectors. If $a=0$, its common zero set has column rank $4+r-b=m-b$. If $a>0$, its common zero set consists of the old intersection $\bigcap J_s$ and all new columns except the $b$ selected coordinates. Its rank is therefore at most
\[
 (4-a)+(r-b)=m-(a+b).
\]
All the required inequalities hold. Proposition~\ref{prop:small-excess} rules out a counterexample with $n\le m+1$, proving the last assertion.
\end{proof}

The positive result for column excess at most one and the counterexamples at excess two locate a precise boundary for this proposed criterion. They do not rule out other combinatorial independence criteria using additional information about the zero pattern. The failure established here is specifically the insufficiency of the ranks in~\eqref{eq:ranktest}.

\section*{Declarations}
\noindent\textbf{Funding.} The author's research is partially supported by an AMS--Simons Research Enhancement Grant (award no.~SFI-MPS-T-Institutes-00021934).

\smallskip
\noindent\textbf{Conflict of interest.} The author declares no conflict of interest.

\smallskip
\noindent\textbf{Data availability.} No datasets were generated or analyzed. All matrices, identities, and proofs needed to verify the results are included in the article.

\smallskip
\noindent\textbf{AI disclosure.} During the preparation of this manuscript, the author used ChatGPT (OpenAI) to assist with language editing and the preparation of code for computational checks. The author takes full responsibility for the accuracy, originality, and integrity of the final manuscript.

\end{document}